\documentclass[a4,11pt]{article}

\usepackage{algpseudocode}% http://ctan.org/pkg/algorithmicx
\usepackage{algorithm,algcompatible} % for algorithm table

\algnewcommand\algorithmicinput{\textbf{Input:}}
\algnewcommand\INPUT{\item[\algorithmicinput]}
\algnewcommand\algorithmicoutput{\textbf{Output:}}
\algnewcommand\OUTPUT{\item[\algorithmicoutput]}
\usepackage{amsmath}% sophisticated mathematical formulas with amstex (includes \text{})
\usepackage{mathtools}% fix amsmath deficiencies
\usepackage{amssymb}% sophisticated mathematical symbols with amstex
\usepackage{amsthm}% theorem environments
\usepackage[american]{babel}%
\usepackage{bm}% for bold math symbols
\usepackage{bbm}% only for indicator functions (replace by dsfont and \mathds{1} if type 3 fonts are not allowed)
\usepackage{booktabs}
\usepackage[top=30truemm,bottom=30truemm,left=25truemm,right=25truemm]{geometry}
\usepackage[bookmarksnumbered=true]{hyperref}
\hypersetup{
     colorlinks = true,
     linkcolor = blue,
     anchorcolor = blue,
     citecolor = blue,
     filecolor = blue,
     urlcolor = blue
}
\usepackage{comment}% for commenting out
\usepackage{color} % for textcolor command
\usepackage{csquotes}% recommended by biblatex
\usepackage{enumitem}
\usepackage{epstopdf}% To incorporate .eps illustrations using PDFLaTeX, etc.
\usepackage{eso-pic}%
\usepackage{etoolbox}%
\usepackage{fancyvrb}% for listings
\usepackage[T1]{fontenc}%
\usepackage{graphicx}% for including figures
\usepackage{grffile}%
\usepackage[
hyperref,%
table%
]{xcolor}%
\usepackage{ifthen}
\newboolean{PrintVersion}
\setboolean{PrintVersion}{false}
\ifthenelse{\boolean{PrintVersion}}{   % for improved print quality, change some hyperref options
\hypersetup{  % override some previously defined hyperref options
    citecolor=black,%
    filecolor=black,%
    linkcolor=black,%
    urlcolor=black}
}{} % end of ifthenelse (no else)

\usepackage{latexsym,indentfirst}
\usepackage{listings}% for including source code
\usepackage{lmodern}%
\usepackage{lscape} % for rotation
\usepackage{multirow}%
\usepackage{microtype}%
\usepackage{rotating}%
\usepackage{scrlayer-scrpage}%
\usepackage{siunitx}%
\usepackage{subfigure}% Support for small, `sub' figures and tables
\usepackage{tabularx}%
\usepackage{threeparttable} %for table partitioned
\usepackage{tikz}% sophisticated graphics package
\usetikzlibrary{positioning}% for relative positioning of nodes
\usetikzlibrary{calc}
\usetikzlibrary{shapes.geometric}
\usetikzlibrary{arrows.meta}
\usepackage{vruler}% for showing line numbers
\usepackage{wrapfig}%

\theoremstyle{definition}
\newtheorem{theorem}{Theorem}

\newtheorem{lemma}{Lemma}
\newtheorem{proposition}{Proposition}
\newtheorem{definition}{Definition}
\newtheorem{example}{Example}
\newtheorem{remark}{Remark}

\xdefinecolor{lightgray}{RGB}{247, 247, 247}%
\xdefinecolor{semilightgray}{RGB}{240, 240, 240}%
\xdefinecolor{middlegray}{RGB}{127, 127, 127}%
\xdefinecolor{blue}{RGB}{58, 95, 205}%
\xdefinecolor{deepskyblue}{RGB}{0, 154, 205}%
\xdefinecolor{chocolate}{RGB}{205, 102, 29}%

\newcommand*{\Unif}{\operatorname{U}}
\newcommand*{\II}{[0,1]}
\newcommand*{\IIo}{(0,1)}
\newcommand*{\IN}{\mathbb{N}}

\newcommand*{\IR}{\mathbb{R}}

\newcommand*{\conv}{\operatorname{conv}}
\newcommand*{\Cov}{\operatorname{Cov}}

\newcommand*{\Prob}{\mathbb{P}}

\newcommand*{\rd}{\mathrm{d}}

\newcommand*{\Var}{\operatorname{Var}}

\newcommand{\bbb}{\bm{b}}

\newcommand{\bu}{\bm{u}}

\newcommand{\bU}{\bm{U}}

\newcommand{\bX}{\bm{X}}

\newcommand{\bone}{\bm{1}}

\newcommand*{\deq}{\ \smash{\omu{\text{\tiny{d}}}{=}{}}\ }
\renewcommand*{\i}{{-1}}

\newcommand{\ou}[3]{%
  \mathrel{%
    \vcenter{\offinterlineskip
      \ialign{##\cr$#1$\cr\noalign{\kern-#3}$#2$\cr}%
    }%
  }%
}
\newcommand*{\omu}[3]{\underset{#3}{\overset{#1}{#2}}}
\newcommand{\sqc}[2]{#1_{1},\dots,#1_{#2}}

\newcommand*{\T}{^{\top}}

\title{Matrix compatibility and correlation mixture representation of generalized Gini's gamma}

\author{Takaaki Koike\footnote{
Graduate School of Economics, Hitotsubashi University, 2-1 Naka, Kunitachi, Tokyo 186-8601, Japan, 
E-mail: takaaki.koike@r.hit-u.ac.jp
} \ and\
Marius Hofert\footnote{
Department of Statistics and Actuarial Science,
Faculty of Science,
The University of Hong Kong,
Pokfulam, Hong Kong,
E-mail: mhofert@hku.hk}
}

\begin{document}
%\gtfamily
\maketitle
\date{}

%%%%%%    TEXT START    %%%%%%
\begin{abstract}
 Representations of measures of
  concordance in terms of Pearson's correlation coefficient are studied.
  All transforms of random variables are characterized such that the correlation
  coefficient of the transformed random variables is a measure of concordance.
  Gini's gamma then is generalized and it is shown that the resulting generalized Gini's gamma can
  be represented as a mixture of measures of concordance that are Pearson's
  correlation coefficients of transformed random variables.
  As an application of this correlation mixture representation of generalized
  Gini's gamma, lower and upper bounds of the compatible set of generalized
  Gini's gamma, the collection of all square matrices of pairwise generalized Gini's gammas, are derived.
  \end{abstract}

% JEL classification and keywords
\hspace{0mm}

\section{Introduction}\label{sec:introduction}
A bivariate random vector $(X,Y)$ is said to be more \emph{concordant} than another random vector $(X',Y')$ if $X$ and $Y$ are more likely to take large (small) values simultaneously than $X'$ and $Y'$ do.
If $X\sim F_X$ and $Y\sim F_Y$ are continuous, concordance is a property of the $\emph{copula}$ $C$ of $(X,Y)$, which is the distribution
function of $(F_X(X),F_Y(Y))$ with standard uniform marginal
distributions.  A measure of concordance quantifies concordance of $(X,Y)$ by a single number in $[-1,1]$ (see Definition~1 for
an axiomatic definition of a measure of concordance and Examples~1 and~3 for
examples).  Measures of concordance are axiomatically required to depend
  only on the underlying copula $C$ of $(X,Y)$. The
  well-known Pearson's correlation coefficient
  $\rho(X,Y)=\Cov(XY)/\sqrt{\Var(X)}\sqrt{\Var(Y)}$ for quantifying linear dependence is therefore not a measure of
  concordance (see Embrechts, McNeil \& Straumann, 2002 for shortcomings of correlation as
  a measure of association).
However, as we will do in Section~\ref{sec:correlation:based:mocs}, these limitations can be overcome by suitably transforming $X$ and $Y$.

For a multivariate random vector $\bX=(\sqc{X}{d})$, $d\geq 3$, we consider a square matrix whose $(i,j)$-entry is the measure of concordance $\kappa$ of $(X_i,X_j)$, for $i,j=1,\dots,d$.
This matrix, called the $\kappa$-\emph{matrix} of a $d$-dimensional copula $C$ associated with $\bX$, summarizes pairwise monotone dependencies of $\bX$ similarly to how the correlation matrix $\rho(\bX)=(\rho(X_i,X_j))$ summarizes pairwise linear dependence of $\bX$.
The range of the mapping from $d$-dimensional random vector to $\kappa$-matrices, that is the collection of all possible $\kappa$-matrices for a given measure of concordance $\kappa$, is called the \emph{compatible set} of $\kappa$ and denoted by $\mathcal R_d(\kappa)$ (see Definition~6 below).

Although the entries of a $\kappa$-matrix are measures of concordance, they cannot take arbitrary values in $[-1,1]$.
For instance, for the 3-dimensional equicorrelation matrix $P(r)=(r_{i,j})$, with $r_{j,j}=1$ and $r_{i,j}=r \in [-1,1]$ for $i,j=1,2,3$ and $i\neq j$, the matrix $P\left(-\frac{1}{2}\right)$ is attainable for Gaussian rank correlation but not for Blomqvist's beta (see Part~2 of Proposition~3 and Example~5 below).
Characterizing a compatible set $\mathcal R_d(\kappa)$ is important to judge whether a given square matrix is attainable as a $\kappa$-matrix.

Denote by $\mathcal P_d$ the set of all $d$-dimensional correlation matrices.
Then it is well-known that $\mathcal P_d$, the compatible set of $\rho$, coincides with the so-called \emph{elliptope}, which is the set of $d\times d$ symmetric positive semi-definite matrices with diagonal entries equal to $1$.
Every $P \in \mathcal P_d$ is attainable by a normal distribution with covariance matrix $P$.
However, not all elements in $\mathcal P_d$ are attainable for non-normal marginals; for example, in the case of non-continuous marginal distributions, the set of $d$-dimensional correlation matrices of symmetric Bernoulli distributions forms the so-called \emph{cut polytope} denoted by $\mathcal P_d^{\rm{B}}$, which is a strict subset of $\mathcal P_d$ for $d \geq 3$ (see Section~5 below for details).
Hofert \& Koike (2019) showed that compatible sets of typical measures of concordance (such as Spearman's rho $\rho_{\text{S}}$, Blomqvist's beta $\beta$ and Gaussian rank correlation $\zeta$) are smaller than $\mathcal P_d$ and larger than $\mathcal P_d^{\rm{B}}$.
More specifically, $\mathcal R_d(\rho_{\text{S}})=\mathcal P_d$ for $d \leq 9$ (Devroye \& Letac, 2015), $\mathcal R_d(\rho_{\text{S}}) \subsetneq\mathcal P_d$ for $d\geq 12$ (Wang, Wang \& Wang, 2019), $\mathcal R_d(\beta)=\mathcal P_d^{\rm{B}}$ and $\mathcal R_d(\zeta)=\mathcal P_d$ (Hofert \& Koike, 2019), and $\mathcal R_d(\tau)=\mathcal P_d^{\rm{B}}$ (McNeil, Ne{\v{s}}lehov{\'a} \& Smith,\ 2022) where $\tau$ is Kendall's tau.
Let us also mention here that Embrechts, Hofert \& Wang\ (2016) characterized the compatible set of tail dependence coefficients although they are not measures of concordance.

Most results concerning the compatible sets mentioned before are derived based on the \emph{correlation representation} of a measure of concordance $\kappa$, given by%
\begin{align}\label{eq:corr:moc:intro}
\kappa(X,Y)=\rho(g_1(U),g_2(V)),\quad
(U,V)=(F_X(X),F_Y(Y)),
\end{align}
for continuous $X\sim F_X$, $Y \sim F_Y$ and two Borel functions $g_1,g_2: (0,1) \rightarrow \IR$. As such, $\kappa$ is obtained by first transforming the original random vector $(X,Y)$ to its ``rank'' $(U,V)=(F_X(X),F_Y(Y))$ and then transforming $(U,V)$ with $(g_1,g_2)$ to ensure that the correlation of $(g_1(U),g_2(V))$ satisfies certain desirable properties.
Such a representation of $\kappa$ is not only a key to studying the compatible set of $\kappa$ but also of great benefit to intuitively understand and explain the construction and ideas behind $\kappa$.
In addition to the interpretability, the correlation representation helps constructing estimators of $\kappa$, investigating their asymptotic properties and analyzing robustness of $\kappa$ (see Raymaekers \& Rousseeuw, 2021).
When $g_1$ and $g_2$ are continuous, Hofert \& Koike (2019) characterized $(g_1,g_2)$ such that $\kappa$ in~\eqref{eq:corr:moc:intro} is a measure of concordance.
On account of the continuity assumption on $g_1$ and $g_2$, this result does not cover, for example, Blomqvist's beta.

As a first contribution of the present article we refine this characterization theorem by relaxing the continuity assumption on $g_1$ and $g_2$ so that Blomqvist's beta is covered. A main result of ours states that the transformed random variables $g_1(U)$ and $g_2(V)$ must follow the same symmetric distribution so that $\kappa$ in~\eqref{eq:corr:moc:intro} is a measure of concordance (see Theorem~2 for the precise statement).
As an implication of this result, the correlation coefficient of
$(X,Y)$ transformed by the wrapping function proposed by Raymaekers \&
Rousseeuw (2021) as a fast robust correlation cannot be a measure of
concordance.

\emph{Gini's gamma} $\gamma$, proposed by Gini (1914), is a measure of
concordance which does not admit a correlation representation, and its
compatible set is not known in the literature.  As surveyed in Genest, Ne{\v{s}}lehov{\'a} \& Ben~Ghorbal\ (2010), Gini's gamma has gained popularity in applications since it is an $L^1$
(Manhattan distance) alternative to Spearman's rho, which is based on the $L^2$
(Euclidean) distance between ranks (see Diaconis \& Graham, 1977 and Nelsen, 1998).  Moreover, $\gamma$ is a symmetrized version of Spearman's footrule
(Spearman, 1906) (see Salama \& Quade, 2001 and Nelsen \& \'Ubeda-Flores, 2004).  Gini's gamma $\gamma$ is also known to have good sample properties,
especially for the test of independence (see Cifarelli \& Regazzini, 1977 and
Luigi Conti \& Nikitin, 1999).

As a second contribution of the present article we generalize Gini's gamma and
show that the resulting generalized Gini's gamma can be written as a mixture of
correlation representations of measures of concordance (see Theorem~3).  As an
application of this correlation mixture representation, we derive upper and
lower bounds of the compatible set of generalized Gini's gamma in Proposition~5,
which is, to the best of our knowledge, the first result concerning compatibility
of Gini's gamma.

This article is organized as follows. In Section~2, we review measures of
concordance, in particular those of degree one and their representations.
Section~3 focuses on correlation representations of measures of concordance
and provides a characterization for such correlation-based measures of
concordance.  In Section~4 we propose the so-called generalized Gini's gamma 
and show that it can be written as a mixture of correlation representations of
measures of concordance.  Section~5 is devoted to deriving bounds of the
compatible set of generalized Gini's gamma as an application of its correlation
mixture representation. Section~6 concludes with possible avenues for future work.

\section{Measures of concordance of degree one}\label{sec:moc:degree:one}

For $d\in \IN$, a $d$-dimensional \emph{copula} ($d$-copula) $C: \II^d\rightarrow \II$ is a $d$-dimensional distribution function with standard uniform univariate marginal distributions.
Denote by $\mathcal C_d$ the set of all $d$-copulas.
We call $C' \in \mathcal C_d$ \emph{at least as concordant as} $C \in \mathcal C_d$, denoted by $C\preceq C'$, if $C(\bu)\leq C'(\bu)$ for all $\bu=(u_1,\dots,u_d) \in \II^d$.
The $d$-dimensional \emph{comonotonicity} and \emph{independence} copulas are defined by $M_d(\bu)=\min(u_j;j=1,\dots,d)$ and $\Pi_d(\bu)=u_1\cdot\ldots\cdot u_d$, respectively.
We omit the subscript $d$ when $d=2$, that is $M=M_2$ and $\Pi=\Pi_2$.
The bivariate \emph{counter-monotonicity} copula is defined by $W(u,v)=\max(u+v-1,0)$.
By the \emph{Fr\'echet--Hoeffding inequalities}, it holds that $W\preceq C \preceq M$ for all $C \in \mathcal C_2$.

Throughout the article, we fix an atomless probability space $(\Omega,\mathcal A,\Prob)$ and identify $\kappa(C)$ for a map $\kappa:\mathcal C_d\rightarrow \IR$ with $\kappa(\bU)$, where $\bU$ is a $d$-dimensional random vector on $(\Omega,\mathcal A,\Prob)$ having distribution function $C$ (denoted by $\bU \sim C$).
The following definition of a measure of concordance is an adapted version of that found in Scarsini (1984).

\begin{definition}{(Measures of concordance)}\label{def:MOC}
A map $\kappa:\mathcal C_2 \rightarrow \IR$ is called a \emph{measure of concordance} if it satisfies the following properties.
\begin{enumerate}[topsep=0cm]
    \item\label{axiom:normalization} {\bf (Normalization)} $\kappa(M)=1$.
    \item\label{axiom:permutation:invariance} {\bf (Permutation invariance)} $\kappa(V,U)=\kappa(U,V)$ for any $(U,V)\sim C \in \mathcal C_2$.
    \item\label{axiom:reflection:symmetry} {\bf (Reflection symmetry)} $\kappa(U,1-V) =-\kappa(U,V)$ for any $(U,V)\sim C \in \mathcal C_2$.
    \item\label{axiom:monotonicity} {\bf (Monotonicity)} $\kappa(C)\le\kappa(C')$ if $C\preceq C'$ for $C,C' \in \mathcal C_2$.
    \item\label{axiom:continuity} {\bf (Continuity)} $\lim_{n\rightarrow \infty}\kappa(C_n)=\kappa(C)$ if $C_n \rightarrow C$ pointwise for $C_n \in \mathcal C_2$, $n\in\IN$, and $C \in \mathcal C_2$.
 \end{enumerate}
\end{definition}

Property~\ref{axiom:normalization} and Property~\ref{axiom:reflection:symmetry} imply that $\kappa(W)=-1$, and thus that $-1 \le \kappa(C)\le 1$ for all $C \in \mathcal C_2$ by Property~\ref{axiom:monotonicity} and the Fr\'echet--Hoeffding inequalities.
Moreover, $\kappa(\Pi)=0$ since $(U,V)\deq (U,1-V)$ for $(U,V)\sim \Pi$.
Therefore, the set of axioms in Definition~1 recovers that found in Scarsini (1984).

In this article, we consider the following class of measures of concordance.

\begin{definition}{(Measures of concordance of degree $k$)}\label{def:moc:degree:k}
A measure of concordance $\kappa$ is \emph{of degree} $k \in \IN$ if $t \mapsto \kappa(tC+(1-t)C')$, $t\in\II$, is a polynomial in $t$ for every $C,\ C' \in \mathcal C_2$, and the supremal degree of such a polynomial over all $C,\ C' \in \mathcal C_2$ is $k$.
\end{definition}

Various characterizations of measures of concordance of degree one are studied in Edwards \& Taylor (2009).  
In the following we introduce one of these characterizations using so-called $D_4$-\emph{invariant} measures.
Let $\mathcal D_4 = \{e,\ \tau,\  \sigma_1,\ \sigma_2, \ \tau\sigma_1,\  \tau\sigma_2,\  \sigma_1\sigma_2,\ \tau\sigma_1\sigma_2 \}$ be the group of symmetries on $\IIo^2$, where $e(u,v)=(u,v)$, $\sigma_1(u,v)=(1-u,v)$, $\sigma_2(u,v)=(u,1-v)$ and $\tau(u,v)=(v,u)$.

\begin{definition}{(Admissible $D_4$-invariant measures)}\label{def:admissible:D4:invariant:measures}
An \emph{admissible $D_4$-invariant} measure $\mu$ is a Borel measure on $(\IIo^2,\mathfrak B(\IIo^2))$ satisfying the following properties:
\begin{enumerate}[topsep=0cm]
\item\label{cond:D4:invariance} {\bf ($D_4$-invariance)} $\mu(\xi(A))=\mu(A)$ for all $\xi \in \mathcal D_4$ and $A\in \mathfrak B(\IIo^2)$; and
\item\label{cond:finiteness:denominator} {\bf (Finiteness)} $0< \int_{\IIo^2}(M-\Pi)\rd \mu < \infty$.
\end{enumerate}
The set of all admissible $D_4$-invariant measures is denoted by $\mathcal M$.
\end{definition}

\begin{theorem}{( Edwards, Mikusi\'nski \& Taylor, 2005 and Edwards \& Taylor, 2009)}\label{thm:moc:degree:one:characterization}
A map $\kappa: \mathcal C_2\rightarrow \IR$ is a measure of concordance of degree one if and only if there exists a unique measure $\mu \in \mathcal M$ such that
\begin{align}\label{eq:moc:degree:one:characterization}
\kappa(C)=\frac{\int_{\IIo^2}(C-\Pi)\,\rd \mu}{\int_{\IIo^2}(M-\Pi)\,\rd \mu}.
\end{align}
\end{theorem}

We denote a measure of concordance of the form in Equation~\eqref{eq:moc:degree:one:characterization} by $\kappa_\mu$, and call $\mu$ the \emph{generating measure} of $\kappa_\mu$.
Note that Edwards \& Taylor (2009) define $\mu$ on $\II^2$. They require $\mu$ to have zero mass on the boundary of $\II^2$ and to satisfy $\int_{\II^2}(M-\Pi)\,\rd \mu=1$ so that $\mu \in \mathcal M$ is uniquely determined.
This difference to our exposition does not affect the above characterization result since mass on the boundary of $\II^2$ does not change the values of the numerator and denominator in Equation~\eqref{eq:moc:degree:one:characterization}, and $\mu \in \mathcal M$ can always be normalized to be $\tilde \mu =\mu/\int_{\II^2}(M-\Pi)\,\rd \mu \in \mathcal M$ by Property~\ref{cond:finiteness:denominator} of Definition~3.
Note that since $\kappa_\mu(tC + (1-t)C')=t \kappa_\mu(C)+(1-t)\kappa_\mu(C')$ for any $t \in \II$ and $C,\ C' \in \mathcal C_2$, measures of concordance of degree one are convex-preserving.
Next we provide examples of measures of concordance of degree one. 

\begin{example}{(Examples of measures of concordance of degree one)}\label{ex:moc:degree:one}
\begin{enumerate}[topsep=0cm]
    \item[1)]\label{example:spearman:rho} \emph{Spearman's rho} (Spearman, 1904): If
    $\mu$ is the probability measure with its mass uniformly distributed on $\IIo^2$, then $\mu$ generates Spearman's rho $\rho_{\text{S}}(C)=12 \int_{\II^2}C\,\rd \Pi -3$.
    \item[2)]\label{example:blomqvist:beta} \emph{Blomqvist's beta} (Blomqvist, 1950):
    The Dirac measure $\mu=\delta_{(\frac{1}{2},\frac{1}{2})}$ generates Blomqvist's beta (also known as \emph{median correlation}) $\beta(C)=4C\left(\frac{1}{2},\frac{1}{2}\right)-1$.
    \item[3)]\label{example:gini:gamma} \emph{Gini's gamma} (Gini, 1914): If $\mu$ is the probability measure of $\frac{1}{2}(M+W)$, that is, if $\mu$ is the probability measure with its mass uniformly distributed on $\{(u,v)\in (0,1)^2: u=v\text{ or } u+v = 1\}$, then $\mu$ generates Gini's gamma $\gamma(C)=4\int_{\II^2}C\,\rd \left(M+W\right)-2$.
\end{enumerate}
\end{example}
Note that Kendall's tau (Kendall, 1938) defined by $\tau(C)=4\int_{[0,1]^2}C(u,v)\,\rd C(u,v)-1$ is a measure of concordance, but of degree two.

Since $M$ is the unique maximal element in $\mathcal C_2$ with respect to $\preceq$, one may want to require a stronger axiom than Property~\ref{axiom:normalization} of Definition~1:
\begin{enumerate}[topsep=0cm]\setcounter{enumi}{5}
    \item\label{axiom:refined:normalization} {\bf (Comonotone uniqueness)} $\kappa(C)=1$ if and only if $C=M$.
\end{enumerate}
See, for example,~Theorem~2.4 and Example~2.7 in~Strothmann, Dette \& Siburg (2022) for a potential application of this requirement.
By Property~\ref{axiom:reflection:symmetry} of Definition~1, Property~\ref{axiom:refined:normalization} implies that $\kappa(C)=-1 \Leftrightarrow C=W$. 
For measures of concordance of degree one, an additional condition is required on the generating measure $\mu$ so that $\kappa_\mu$ satisfies Property~\ref{axiom:refined:normalization}.

\begin{proposition}{(Generating measures for the comonotone uniqueness)}
\label{prop:refined:normalization:axiom:generating:measure}
Let $\kappa_\mu$ be a measure of concordance of degree one with generating measure $\mu \in \mathcal M$.
Then $\kappa_\mu$ satisfies Property~\ref{axiom:refined:normalization} if and only if the support of $\mu$ contains the main diagonal of $(0,1)^2$.
\end{proposition}

\begin{proof}%
Notice that $\kappa_\mu(C)=1$ if and only if $\int_{\IIo^2}(M-C)\rd \mu=0$ from Representation (2) in Theorem 1.

To show the necessity of Property~6, suppose that the support of $\mu$ contains the main diagonal of $(0,1)^2$, and that $\kappa_\mu(C)=1$ for $C \in \mathcal C_2$.
As seen in Example~2.6.4 of Durante \& Sempi (2006), $C=M$ holds if and only if $C(t,t)=t$ for all $t \in (0,1)$. 
In view of this, suppose that there exists $t_0 \in (0,1)$ such that $C(t_0,t_0)<M(t_0,t_0)=t_0$.
Since $t \mapsto h(t):=M(t,t)-C(t,t)$ is continuous on $(0,1)$, there exists $\epsilon>0$ such that $h$ is strictly positive on $(t_0-\epsilon,t_0+\epsilon)\subseteq (0,1)$.
By assumption, we have that 
 $\mu((t_0-\epsilon,t_0+\epsilon)^2)>0$.
 Together with $\kappa_\mu(C)=1$ and $C\preceq M$, it holds that
\begin{align*}
0=\int_{\IIo^2}(M-C)\rd\mu \geq \int_{(t_0-\epsilon,t_0+\epsilon)^2}(M-C)\rd\mu >0,
\end{align*}
which is a contradiction.

We show the sufficiency of Property~6 by contraposition, that is, if the support of $\mu$ does not contain the main diagonal of $(0,1)^2$, then $\kappa_\mu(C)=1$ does not imply $C=M$.
To this end, assume that the support of $\mu$ does not contain $(t_0,t_0)$ for some $t_0 \in (0,1)$.
Then there exists $\epsilon>0$ such that $\mu((t_0-\epsilon,t_0+\epsilon)^2)=0$ for $(t_0-\epsilon,t_0+\epsilon)\subseteq(0,1)$.
Let $C_\epsilon \in \mathcal C_2$ denote the ordinal sum of $\{M,W,M\}$ with respect to the grid $\{0,t_0-\epsilon,t_0+\epsilon,1\}$ (see Section~3.2.2 of Nelsen, 2006 for the definition of ordinal sums).
Then it is straightforward to check that $C_\epsilon\neq M$ on $(t_0-\epsilon,t_0+\epsilon)^2$ and $C_\epsilon=M$ elsewhere.
Therefore, we have that
\begin{align*}
\int_{\IIo^2}(M-C_\epsilon)\rd\mu =0,
\end{align*}
and thus $\kappa_\mu(C_\epsilon)=1$ for $C_\epsilon\neq M$.
\end{proof}

\section{Correlation-based measures of concordance}\label{sec:correlation:based:mocs}
For a distribution function $G:\IR\rightarrow \II$, its \emph{quantile
function} $G^\i:\IIo\rightarrow \IR$ is defined by $G^\i(p)=\inf\{x\in\IR:G(x)\ge p\}$, $p\in \IIo$.
A distribution function $G$ with a finite mean is called $\emph{symmetric}$ if $X-\mathbb{E}[X]\deq \mathbb{E}[X]-X$ for $X\sim G$.

\begin{definition}{(Concordance-inducing distributions)}
  A distribution function $G:\IR\rightarrow \II$ is \emph{concordance-inducing}
  if it is nondegenerate, symmetric and has a finite second moment.  The set of
  all concordance-inducing distribution functions is denoted by $\mathcal G$.
\end{definition}

For distribution functions $G,\ G_1$ and $G_2$, denote by 
$\lambda_G$ and $\lambda_{G_1,G_2}$ the push-forward Lebesgue measures $\lambda_G(A)=\lambda_1\{x\in \IR: G(x)\in A\}$, $A\in \mathfrak B(\IIo)$, and $\lambda_{G_1,G_2}(A)=\lambda_2\{(x,y)\in \IR^2: (G_1(x),G_2(y))\in A\}$, $A\in \mathfrak B(\IIo^2)$, respectively, where $\lambda_d$, $d\in \IN$, is the Lebesgue measure on $\IR^d$.
The following proposition states that $\lambda_{G,G}$ is a generating measure in the sense of Theorem~1 if $G$ is concordance-inducing.

\begin{proposition}{(Generating measures of correlation-based measures of concordance)}\label{prop:generating:measure:correlation:based:mocs}
For any concordance-inducing distribution $G \in \mathcal G$, the push-forward Lebesgue measure $\lambda_{G,G}$ is in $\mathcal M$, and the measure of concordance generated by $\lambda_{G,G}$ is given by
\begin{align}\label{eq:G:transformed:rank:correlation}
\kappa_{\lambda_{G,G}}(C)=\rho(G^\i(U),G^\i(V)),\quad (U,V)\sim C,
\end{align}
where $\rho$ is \emph{Pearson's correlation coefficient}.
\end{proposition}

We call the map $\kappa_{\lambda_{G,G}}$ in
Equation~\eqref{eq:G:transformed:rank:correlation} the \emph{$G$-transformed
  rank correlation}.

\begin{proof}
For a concordance-inducing distribution function $G \in \mathcal G$, we first show that $\lambda_{G,G}$ satisfies Properties~\ref{cond:D4:invariance} and~\ref{cond:finiteness:denominator} of Definition~3.

To show Property~\ref{cond:D4:invariance}, the $D_4$-invariance of $\lambda_{G,G}$, it suffices to show that $\lambda_G(\sigma(A))=\lambda_G(A)$,  $A\in \mathfrak B(\IIo)$, for $\sigma(x)=1-x$, $x\in (0,1)$, since $\lambda_{G,G}$ is the product measure $\lambda_{G}\times \lambda_{G}$.
By symmetry of $G$, it holds that $G(x+\mu)=1-G((\mu-x)-)$ for all $x\in \IR$, where $\mu$ is the mean of $G$.
Note that the set of discontinuity points of $G$ is at most countable.
Together with the invariance of the Lebesgue measure $\lambda_1$ under isometry, it holds that
\begin{align*}
\lambda_{G}(\sigma(A))&=\lambda_1\{x\in \IR: G(x)\in \sigma(A)\}\\
& =\lambda_1\{2\mu - \tilde x\in \IR: G(\tilde x)\in A\}\\
&= \lambda_1\{\tilde x\in \IR: G(\tilde x)\in A\}=\lambda_{G}(A).
\end{align*}

To show the finiteness of the integral in Property~\ref{cond:finiteness:denominator}, we define the map $f:(x,y)\mapsto (G(x),G(y))$, $(x,y)\in \IR^2$ for notational convenience.
By Theorem 3.6.1\ in Bogachev (2007), $\int_{\IIo^2}(C-\Pi)\,\rd (\lambda_2\circ f^\i) <\infty$ if and only if  $\int_{\IR^2}(C-\Pi)\circ f \,\rd \lambda_2 <\infty$.
Moreover, under any of these integrability conditions, it holds that $\int_{\IIo^2}(C-\Pi)\,\rd (\lambda_2\circ f^\i) =\int_{\IR^2}(C-\Pi)\circ f \,\rd \lambda_2$. 
The latter formula holds for any $C \in \mathcal C_2$ since
Hoeffding's identity (see Lemma~2 in Lehmann, 1966) leads to
\begin{align*}
\int_{\IR^2}(C-\Pi)\circ f \,\rd \lambda_2 &= \int_{\IR^2}\left\{C(G(x),G(y))-\Pi(G(x),G(y))\right\}\,\rd \lambda_2(x,y)\\
&= \Cov(G^\i(U),G^\i(V)),\quad (U,V)\sim C,\\
&\leq \Var(G^\i(U)) <\infty,
\end{align*}
where the last inequality follows from the Cauchy--Schwarz inequality.
By taking $C=M$, we have that $\int_{\IIo^2}(M-\Pi)\,\rd \lambda_{G,G}<\infty$.
To show that $\int_{\IIo^2}(M-\Pi)\,\rd \lambda_{G,G}>0$, we use the property
$\rho(G^\i(U),G^\i(U))=1$ for $U \sim \Unif(0,1)$ (see Part~3 of Theorem~4 in Embrechts, McNeil \& Straumann, 2002).
Since $G \in \mathcal G$ is nondegenerate, Hoeffding's identity implies that
\begin{align*}
\int_{\IIo^2}(M-\Pi)\,\rd \lambda_{G,G}&=\int_{\IR^2}(M-\Pi)\circ f \,\rd \lambda_2
= \Cov(G^\i(U),G^\i(U))\\
& = \rho(G^\i(U),G^\i(U))\Var(G^\i(U)) \\
&=\Var(G^\i(U))>0.
\end{align*}

Finally, Representation~\eqref{eq:G:transformed:rank:correlation} is obtained via
\begin{align*}
\kappa_{\lambda_{G,G}}(C)&=\frac{\int_{\IIo^2}(C-\Pi)\,\rd \lambda_{G,G}}{\int_{\IIo^2}(M-\Pi)\,\rd \lambda_{G,G}}=\frac{\Cov(G^\i(U),G^\i(V))}{\Cov(G^\i(U),G^\i(U))}\\
&= \frac{\Cov(G^\i(U),G^\i(V))}{\sqrt{\Var(G^\i(U))}\sqrt{\Var(G^\i(V))}}=\rho(G^\i(U),G^\i(V)).
\end{align*}
\end{proof}

\begin{example}{(Examples of $G$-transformed rank correlations)}
Spearman's rho $\rho_{\text{S}}$, Blomqvist's beta $\beta$, and \emph{Gaussian rank correlation} $\zeta$ (Sidak, Sen \& Hajek,\ 1999 and Boudt, Cornelissen \& Croux,\ 2012; also known as \emph{van der Waerden's coefficient} or \emph{normal score correlation}) are all examples of $G$-transformed rank correlations, which have standard uniform, symmetric Bernoulli, and standard Gaussian distributions, respectively, as concordance-inducing distributions (see Example~1 in Hofert \& Koike, 2019).
\end{example}

Note that Gini's gamma is not a $G$-transformed rank correlation since the probability measure of $\frac{1}{2}(M+W)$ is not a push-forward Lebesgue measure of the form $\lambda_{G,G}$.

For two Borel functions $g_1, g_2:\IIo\rightarrow \IR$, let
\begin{align}\label{eq:g1:g2:transformed:correlation}
\rho_{g_1,g_2}(C)=\rho(g_1(U),g_2(V)),\quad (U,V)\sim C.
\end{align}
We also write $\rho_g$ when $g_1=g_2=g$ for some $g:\IIo\rightarrow \IR$.
With this notation, the $G$-transformed rank correlation~\eqref{eq:G:transformed:rank:correlation} can be written as $\rho_{G^\i}$.
The next theorem states that, under some mild assumptions on $g_1$ and $g_2$, the case $g_1=g_2=G^\i$ for $G \in \mathcal G$ covers all measures of concordance of the form given in Equation~\eqref{eq:g1:g2:transformed:correlation}.

\begin{theorem}{(Characterization of correlation-based measures of concordance)}
\label{thm:characterization:correlation:based:mocs}
Let $g_1,g_2:\IIo\rightarrow \IR$ be left-continuous functions whose discontinuity points form a discrete set.
Then the map $\rho_{g_1,g_2}(C)=\rho(g_1(U),g_2(V))$, $(U,V)\sim C$, is a measure of concordance if and only if $\rho_{g_1,g_2}=\rho_{G^\i}$ for some $G \in \mathcal G$.
\end{theorem}

\begin{proof}

For any concordance-inducing distribution $G \in \mathcal G$, the map $\rho_{g_1,g_2}=\rho_{G^\i}$ is a measure of concordance by Proposition~2.

To show the converse, let $u,\,u' \in \IIo$, $u<u'$ be any continuity points of $g_1$ and $v,\, v' \in \IIo$, $v<v'$, be any continuity points of $g_2$.
Since the discontinuity points of $g_1$ and those of $g_2$ are isolated, there exists a sufficiently large $N \in \IN$ and indices $i,i',j,j' \in \{1,\dots,N-1\}$ such that
  \begin{align*}
\frac{i-1}{N}<u \leq \frac{i}{N},\quad
\frac{i'-1}{N}<u' \leq \frac{i'}{N},\quad
\frac{j-1}{N}<v \leq \frac{j}{N},\quad
\frac{j'-1}{N}<v' \leq \frac{j'}{N}
  \end{align*}
  with $(\frac{i-1}{N},\frac{i}{N}]\cap (\frac{i'-1}{N},\frac{i'}{N}]=\emptyset$, $(\frac{j-1}{N},\frac{j}{N}]\cap (\frac{j'-1}{N},\frac{j'}{N}]=\emptyset$ and such that $g_1$ and $g_2$ are continuous on $(\frac{i-1}{N},\frac{i}{N}]\cup (\frac{i'-1}{N},\frac{i'}{N}]$ and $(\frac{j-1}{N},\frac{j}{N}]\cup (\frac{j'-1}{N},\frac{j'}{N}]$, respectively.
 Since the function $(x,y)\mapsto g_1(x)g_2(y)$ is continuous on  $(\frac{i-1}{N},\frac{i}{N}]\cup (\frac{i'-1}{N},\frac{i'}{N}] \times (\frac{j-1}{N},\frac{j}{N}]\cup (\frac{j'-1}{N},\frac{j'}{N}]$, we have that
\begin{align}\label{eq:monotonicity:g1:g2}
\left\{g_1(u')-g_1(u)\right\}\left\{g_2(v')-g_2(v)\right\}\geq 0,
\end{align}
which can be established by following the proof of Proposition~1 in Hofert \& Koike (2019).
Discontinuity points of $g_1$ and $g_2$ are at most countable (see also Theorem~5.61 in Thomson, Judith \& Andrew,\ 2008 or Proposition~2 in Yim-Ming, 1965), and thus Inequality~\eqref{eq:monotonicity:g1:g2} also holds for countably many discontinuity points of $g_1$ and $g_2$ by considering increasing sequences of continuity points of $g_1$ and $g_2$ approaching their discontinuity points from the left.
Therefore, $g_1$ and $g_2$ are monotone in the sense that $\{g_1(u')-g_1(u)\}\{g_2(v')-g_2(v)\}\geq 0$ for any $0<u<u' < 1$ and $0<v<v'<1$.

By monotonicity of $g_1$ and $g_2$, it holds that $g_1$ and $g_2$ are both non-decreasing, or both non-increasing.
First consider the former case.
For $C \in \mathcal C_2$ and $(U,V) \sim C$, let $X=g_1(U)\sim G_1$ and $Y=g_2(V)\sim G_2$.
By Lemma A. 23 in F\"ollmer \& Schied (2011), we have that $(G_1^\i(u),G_2^\i(v))=(g_1(u),g_2(v))$, $(u,v)\in (0,1)^2$ almost everywhere (the two functions actually coincide on $(u,v)\in (0,1)^2$ by left-continuity of $G_1$, $G_2$, $g_1$ and $g_2$).
Therefore, we have that $\rho(g_1(U),g_2(V))=\rho(G_1^\i(U),G_2^\i(V))
$.
Finally one can take $G_1=G_2=G$ for some $G \in \mathcal G$ by Theorem~1 in Hofert \& Koike (2019).
In the case where $g_1$ and $g_2$ are non-increasing, the problem reduces to the above case since $\rho(g_1(U),g_2(V))=\rho(-g_1(U),-g_2(V))$ for non-decreasing functions $\tilde g_1 = -g_1$ and $\tilde g_2 = -g_2$ by the scale invariance of $\rho$.
\end{proof}

\begin{remark}{(Remarks on Theorem~2)}
\begin{enumerate}[topsep=0cm]
  \item Since values of $g_1$ and $g_2$ on countable sets do not affect the value of $\rho_{g_1,g_2}(C)$, the same result as in Theorem~2 holds when $g_1$ or $g_2$ are right-continuous, or more generally, when they are continuous except on a discrete (and thus at most countable) set.
 As mentioned in the proof, left-continuity of $g_1$ and $g_2$ leads to $g_1=G^\i$ and $g_2 = G^\i$ on $(0,1)^2$.
 These equalities hold $\lambda_1$-almost everywhere in the aforementioned more general case.
 \item The correlation representation $\rho_{G^\i}$ is not unique since the location-scale invariance of the Pearson's correlation coefficient $\rho$ implies that
$\rho_{G_{\mu_1,\sigma_1}^\i,G_{\mu_2,\sigma_2}^\i}=\rho_{G^\i}$ for any $\mu_1,\ \mu_2 \in\IR$ and $\sigma_1,\ \sigma_2>0$, where $G_{\mu,\sigma}(x)=G\left(\frac{x-\mu}{\sigma}\right)$, $x \in \IR$.
\end{enumerate}
\end{remark}

\section{Correlation mixture representation of Gini's gamma}\label{sec:correlation:representation:gini:gamma}

In this section we show that Gini's gamma can be written as a mixture of $G$-transformed rank correlations.
We first propose the following class of measures of concordance.

\begin{definition}{(Generalized Gini's gamma)}\label{def:generalized:gini:gamma}
A measure of concordance is called a \emph{generalized Gini's gamma} if it admits the representation in Equation~\eqref{eq:moc:degree:one:characterization} for some $\mu \in \mathcal M$, and the mass of the generating measure $\mu$ is concentrated on $\{(u,v)\in (0,1)^2: u=v\text{ or } u+v = 1\}$.
\end{definition}

By $D_4$-invariance (Property~1 of Definition~3) and Proposition~1, a measure of concordance of degree one $\kappa_\mu$, $\mu \in \mathcal M$, satisfies Property~\ref{axiom:refined:normalization} if and only if the support of $\mu$ contains the main and the secondary diagonals.
Therefore, a generalized Gini's gamma has a generating measure whose mass is concentrated on the minimal support required for the measure to satisfy Property~\ref{axiom:refined:normalization}. 

\begin{example}{(Examples of generalized Gini's gamma)}\label{ex:generalized:gini:gamma}
\begin{enumerate}[topsep=0cm]
\item\label{ex:gini:gamma} \emph{Gini's gamma}: Since the generating measure of Gini's gamma $\gamma$ has its mass uniformly distributed on the main and secondary diagonals, $\gamma$ is a generalized Gini's gamma.
\item\label{ex:generalized:blomqvist:beta} \emph{Generalized Blomqvist's beta}: For $p \in \left(0,\frac{1}{2}\right]$, denote by $\beta_p$ a measure of concordance generated by
\begin{align}\label{eq:generating:measure:generalized:blomqvist:beta}
\mu_p = \delta_{(p,p)}+\delta_{(p,1-p)}+\delta_{(1-p,p)}+\delta_{(1-p,1-p)}.
\end{align}
We call $\beta_p$ \emph{generalized Blomqvist's beta} since $p=\frac{1}{2}$ reduces to Blomqvist's beta.
\end{enumerate}
\end{example}

\begin{lemma}{(Correlation representation of generalized Blomqvist's beta)}\label{lemma:correlation:generalized:blomqvist}
For $0<p\le \frac{1}{2}$, generalized Blomqvist's beta $\beta_p$ can be represented as
\begin{align*}
\beta_p(C)=\rho(G^\i(U;p),G^\i(V;p))
\end{align*}
for $(U,V)\sim C$, where
\begin{align}\label{gp:function}
G(x;p)=p\bone_{\{x\geq -1\}}+(1-2p)\bone_{\{x\geq 0\}}+p\bone_{\{x\geq 1\}}.
\end{align}
\end{lemma}

\begin{proof}
Let $X_p$ be a discrete random variable taking on $-1$, $0$ and $1$ with probabilities $\Prob(X_p=-1)=p$, $\Prob(X_p=0)=1-2p$ and $\Prob(X_p=1)=p$, respectively.
Then $X_p \sim G(\cdot;p)$ and its quantile function is given by
\begin{align}\label{eq:generalized:inverse:three:points:distribution}
G^\i(u;p)=\begin{cases}
-1, & \text{ if } 0< u \le p,\\
0, & \text{ if } p <u \le 1-p,\\
1, & \text{ if } 1-p<u < 1.
\end{cases}
\end{align}

For $C \in \mathcal C_2$ and $(U,V)\sim C$, we have that
\begin{align*}
\mathbb{E}[G^\i(U;p)G^\i(V;p)]&= 
\Prob\left\{(U,V)\in [0,p]^2\right\} \\
&\quad -\Prob\left\{(U,V)\in [0,p]\times (1-p,1]\right\}\\
&\quad -\Prob\left\{(U,V)\in (1-p,1]\times [0,p]\right\}\\
&\quad +\Prob\left\{(U,V)\in (1-p,1]\times (1-p,1]\right\}\\
&= C(p,p) - (p-C(p,1-p)) - (p-C(1-p,p)) \\
&\quad + 1-(1-p)-(1-p)+C(1-p,1-p)\\
&=C(p,p) + C(p,1-p) \\
&\quad +C(1-p,p) + C(1-p,1-p)-1\\
&=\sum_{\xi \in \{e,\sigma_1,\sigma_2,\sigma_1\sigma_2\}}C\circ\xi(p,p)-1,
\end{align*}
where $G^\i(\cdot;p)$ is as in Equation~\eqref{eq:generalized:inverse:three:points:distribution}.
Since $\Var(X_p)=2p$, we have that
\begin{align}
\nonumber&\phantom{={}} \sum_{\xi \in \{e,\sigma_1,\sigma_2,\sigma_1\sigma_2\}}(C-\Pi)\circ\xi(p,p)\\
\nonumber&=\sum_{\xi \in \{e,\sigma_1,\sigma_2,\sigma_1\sigma_2\}}C\circ\xi(p,p)-
\sum_{\xi \in \{e,\sigma_1,\sigma_2,\sigma_1\sigma_2\}}\Pi\circ\xi(p,p)\\
\nonumber&=\mathbb{E}[G^\i(U;p)G^\i(V;p)]- \mathbb{E}[G^\i(U;p)]\mathbb{E}[G^\i(V;p)]\\
\nonumber&= \Var(X_p)\times\rho(G^\i(U;p),G^\i(V;p))\\
\label{eq:auxiliary:correlation:Gp}&=2p\times\rho(G^\i(U;p),G^\i(V;p)).
\end{align}
Since $G(\cdot;p)$ is symmetric for any $p \in \left(0,\frac{1}{2}\right]$, we have that $\rho_{G^\i(\cdot;p)}(M)=1$; see Part~3 of Theorem~4 in Embrechts, McNeil \& Straumann\ (2002).
Therefore, we have that
\begin{align*}
\sum_{\xi \in \{e,\sigma_1,\sigma_2,\sigma_1\sigma_2\}}(M-\Pi)\circ\xi(p,p)=2p,
\end{align*}
and thus that
\begin{align*}
\beta_p(C)=\frac{\sum_{\xi \in \{e,\sigma_1,\sigma_2,\sigma_1\sigma_2\}}(C-\Pi)\circ\xi(p,p)}{\sum_{\xi \in \{e,\sigma_1,\sigma_2,\sigma_1\sigma_2\}}(M-\Pi)\circ\xi(p,p)}=\rho(G^\i(U;p),G^\i(V;p)).
\end{align*}
\end{proof}

A generator of generalized Gini's gamma can be regarded as a mixture of generators of generalized Blomqvist's beta's with respect to $p\in (0,\frac{1}{2}]$.
This interpretation leads to the following representation of generalized Gini's gamma in terms of Pearson's correlation coefficient.

\begin{theorem}{(Correlation mixture representation of generalized Gini's gamma)}\label{thm:correlation:representation:generalized:gini:gamma}
Generalized Gini's gamma can be represented as
\begin{align}\label{eq:correlation:representation:generalized:gini:gamma}
\gamma_\nu(C)=\int_{\left(0,\frac{1}{2}\right]}\rho(G^\i(U;p),G^\i(V;p))\,\rd \nu(p),
\end{align}
for some unique Borel probability measure $\nu$ on $\left(0,\frac{1}{2}\right]$, where $(U,V)\sim C$ and $G$ is given in~\eqref{gp:function}.
\end{theorem}

\begin{proof}
Let $\kappa$ be a generalized Gini's gamma.
Since $\kappa$ is a measure of concordance of degree one, Theorem~1 in Edwards \& Taylor~(2009) implies that there exists a unique generating measure $\mu$ such that $\int_{\IIo^2}(M-\Pi)\,\rd \mu=1$.
Let $\Gamma=\left\{(u,v)\in\left(0,\frac{1}{2}\right): u=v\right \}$.
Since $\mu$ is $D_4$-invariant and concentrated on the main and secondary diagonals of $(0,1)^2$,
we have that
\begin{align*}
\kappa(C)&=\int_{(0,1)^2}(C-\Pi)\,\rd  \mu\\
&=\sum_{\xi \in \{e,\sigma_1,\sigma_2,\sigma_1\sigma_2\}}\int_{\xi(\Gamma)}(C-\Pi)\,\rd \mu
+ (C-\Pi)\left(\frac{1}{2},\frac{1}{2}\right)\, \mu\left\{\left(\frac{1}{2},\frac{1}{2}\right)\right\}
\\
&=\sum_{\xi \in \{e,\sigma_1,\sigma_2,\sigma_1\sigma_2\}}\left\{
\int_{\Gamma}(C-\Pi)\circ \xi\,\rd \mu\right.\\
& \left. \phantom{\sum_{\xi \in \{e,\sigma_1,\sigma_2,\sigma_1\sigma_2\}}\left\{
\int_{\Gamma}(C
\right.}
+ \frac{1}{4}\, (C-\Pi)\circ\xi\left(\frac{1}{2},\frac{1}{2}\right) \,\mu\left\{\left(\frac{1}{2},\frac{1}{2}\right)\right\}
\right\}\\
&=\int_{\left(0,\frac{1}{2}\right]}\biggl(\sum_{\xi \in \{e,\sigma_1,\sigma_2,\sigma_1\sigma_2\}}(C-\Pi)\circ\xi(p,p)\biggr)\,\rd \tilde \nu(p),
\end{align*}
where $\tilde \nu$ is a Borel measure on $\left(0,\frac{1}{2}\right]$ determined via
\begin{align*}
\tilde \nu((0,p])=
\mu((0,p]^2\cap \Gamma)+
 \frac{1}{4}
  \mu\left\{\left(\frac{1}{2},\frac{1}{2}\right)\right\}
  \delta_{\frac{1}{2}}((0,p]^2).
\end{align*}
With $\nu$ being a Borel measure defined via $\rd \nu(p) = 2p \,\rd \tilde \nu(p)$, Equation~\eqref{eq:auxiliary:correlation:Gp} yields
\begin{align*}
\kappa(C)=\int_{\left(0,\frac{1}{2}\right]}\rho_{G^\i(\cdot;p)}(C) \,\rd \nu(p).
\end{align*}
Together with $\kappa(M)=1$, we have that $1=\nu(\left(0,\frac{1}{2}\right])$, that is, $\nu$ is a probability measure on $\left(0,\frac{1}{2}\right]$.
Since the normalized generating measure $\mu$ is unique, the probability measure $\nu$ is also unique by its construction.
\end{proof}

\begin{example}{(Examples of $\nu$-measures)}
 \begin{enumerate}[topsep=0cm]
\item\label{ex:gini:gamma:nu:measure} \emph{Gini's gamma}: The generating measure of Gini's gamma $\gamma$ has a mass of 8 uniformly on the two diagonals.
Therefore, the corresponding $\nu$-measure has a linear probability density function $f_\nu(p)=8p$, $p \in \left(0,\frac{1}{2}\right]$.
\item\label{ex:generalized:blomqvist:beta} \emph{Generalized Blomqvist's beta}: For $p \in \left(0,\frac{1}{2}\right]$, the generalized Blomqvist's beta $\beta_p$ has the generating measure~\eqref{eq:generating:measure:generalized:blomqvist:beta}.
Therefore, the corresponding $\nu$-measure has a point mass 1 at $p$.
\end{enumerate}
\end{example}

\section{Compatibility bounds for Gini's gamma}\label{sec:compatibility:bounds:gini:gamma}

As an application of Theorems~2 and~3, in this section we derive compatibility bounds for $G$-transformed rank correlation and generalized Gini's gamma.

For $C \in \mathcal C_d$ and $(i,j)\in \{1,\dots,d\}^2$, $i\neq j$, denote by $C_{i,j}$ the $(i,j)$-marginal copula of $C$.
Then the $d$-dimensional square matrix $(\kappa(C_{i,j}))_{i,j=1,\dots,d}$ is called the $\kappa$-\emph{matrix} of $C \in \mathcal C_d$.

\begin{definition}{(Compatible set)}
\label{def:compatible:set}
For $d \in \IN$ and a measure of concordance $\kappa:\mathcal C_2 \rightarrow [-1,1]$, the range
\begin{align*}
\mathcal R_d(\kappa) = \{(\kappa(C_{i,j}))_{i,j=1,\dots,d}: C \in \mathcal C_d\}
\end{align*}
is called a $d$-dimensional \emph{compatible set} of $\kappa$.
\end{definition}

Recall that $\mathcal P_d$ and $\mathcal P_d^{\rm{B}}$ represent elliptope and cut polytope, i.e., the set of all $d$-dimensional correlation matrices and the set of all correlation matrices of $d$-dimensional random vectors with symmetric Bernoulli margins, respectively.
The next proposition says that $\mathcal P_d$ and $\mathcal P_d^{\rm{B}}$ are the upper and lower bounds of the compatible sets of $G$-transformed rank correlations.

\begin{proposition}{(Hofert \& Koike, 2019)}
\label{prop:compatibility:bounds:correlation:MOC}
\begin{enumerate}[topsep=0cm]
\item\label{point:prop:compatibility:bounds:correlation:moc} $\mathcal P_d^{\rm{B}}\subseteq \mathcal R_d(\rho_{G^\i}) \subseteq \mathcal P_d$ for all $G \in \mathcal G$.
\item\label{point:prop:attainability:compatibility:bounds:correlation:moc} $\mathcal R_d(\zeta)=\mathcal P_d$ and $\mathcal R_d(\beta)=\mathcal P_d^{\rm{B}}$.
\end{enumerate}
\end{proposition}

It is shown in Huber \& Mari\'c (2015, 2019) that $\mathcal P_d^{\rm{B}}$ is the convex hull given by
\begin{align*}
\mathcal P_d^{\rm{B}}=\conv\{P^{(\bbb)}: \bbb  \in \{0,1\}^d,\ b_1=1\},
\end{align*}
where $P^{(\bbb)}=(2\bbb-1)(2\bbb-1)\T$, that is $\mathcal P_d^{\rm{B}}$ is a convex polytope with $2^{d-1}$ vertices whose off-diagonal entries are all $1$ or $-1$.
For any measure of concordance $\kappa$, it holds that $\kappa(M)=1$ and $\kappa(W)=-1$, and thus
\begin{align*}
P^{(\bbb)} = (\kappa(\bU^{(\bbb)}_{i,j}))_{i,j=1,\dots,d}\quad\text{for}\quad \bU^{(\bbb)}= U \bbb + (1-U)(\bone_d - \bbb),
\end{align*}
where $\bbb \in \{0,1\}^d$, $U \sim \Unif(0,1)$ and $\bone_d=(1,\dots,1)\in \IR^d$.
As a consequence, the lower bound in Part~\ref{point:prop:compatibility:bounds:correlation:moc} of Proposition~3 remains valid for any measure of concordance of degree one since they are convex-preserving.

\begin{proposition}{(Lower compatibility bound of $\kappa_{\mu}$)}
\label{prop:lower:compatibility:bound:degree:one:moc}
Let $\kappa$ be any measure of concordance of degree one.
Then $\mathcal P_d^{\rm{B}} \subseteq \mathcal R_d(\kappa)$.
\end{proposition}

The next proposition states that the upper bound in Part~\ref{point:prop:compatibility:bounds:correlation:moc} of Proposition~3 remains valid for generalized Gini's gamma.

\begin{proposition}{(Upper compatibility bound for $\gamma_\nu$)}
\label{prop:compatibility:bounds:generalized:gini:gamma}
Let $\gamma_\nu$ be a generalized Gini's gamma.
Then $\mathcal P_d^{\rm{B}} \subseteq \mathcal R_d(\gamma_\nu) \subseteq \mathcal P_d$.
\end{proposition}

\begin{proof}
Since $\mathcal P_d^{\rm{B}} \subseteq \mathcal R_d(\gamma_\nu)$ holds by Proposition~4, we will show that $\mathcal R_d(\gamma_\nu)\subseteq \mathcal P_d$.
Let $\nu$ be the probability measure such that the correlation representation~\eqref{eq:correlation:representation:generalized:gini:gamma} holds for $\gamma_\nu$.
Since the integrand $p \rightarrow \rho(G^\i(U;p),G^\i(V;p)) \in [-1,1]$ is continuous and bounded, it holds that
\begin{align*}
\gamma_{\nu}(C)=\lim_{n\rightarrow \infty} \gamma_{\nu^{[n]}}(C)
=\lim_{n\rightarrow \infty}
\sum_{k=1}^n w_k \rho_{G^\i(\cdot;p_k)}(C),
\end{align*}
where $0=p_0<p_1<\cdots<p_n=\frac{1}{2}$,  $w_k = \nu((p_{k-1},p_k])\geq 0$ and $\nu^{[n]}=\sum_{k=1}^n w_k \delta_{p_k}$.
Since $\sum_{k=1}^n w_k =\nu((0,\frac{1}{2}])=1$, $\nu^{[n]}$ is a probability measure.
Moreover, since $G(\cdot;p) \in \mathcal G$, we have that $\mathcal R_d(\rho_{G^\i(\cdot;p_k)})\subseteq \mathcal P_d$ by Part~\ref{point:prop:compatibility:bounds:correlation:moc} of Proposition~3.
Finally, we have that $\mathcal R_d(\gamma_{\nu^{[n]}})\subseteq \mathcal P_d$ for any $n \in \IN$ by convexity of $\mathcal P_d$, and that $\mathcal R_d(\gamma_{\nu})\subseteq \mathcal P_d$ by closedness of $\mathcal P_d$.
\end{proof}

$\mathcal R_d(\gamma_\nu)$ is in general not equal to $\mathcal P_d^{\rm{B}}$ as seen in the following example.

\begin{example}{($\mathcal P_d^{\rm{B}} \neq \mathcal R_d(\gamma_\nu)$)}\label{ex:strictness:lower:bound:generalized:gini:gamma:compatibility}
We construct a $3$-dimensional correlation matrix $P$ such that $P \notin  \mathcal P_3^{\rm{B}}$ and $P \in \mathcal R_3(\gamma)$ for Gini's gamma.
Denote by $P(r)$, $r\in[-1,1]$, the 3-dimensional correlation matrix whose off-diagonal entries are all equal to $r$.
As seen in Example~4 of Hofert \& Koike (2019), we have that $P(r) \in \mathcal P_3$ if and only if $-\frac{1}{2}\leq r \leq 1$, and $P(r) \in \mathcal P_d^{\rm{B}}$ if and only if $-\frac{1}{3}\leq r \leq 1$.
Moreover, Meyer (2013) showed that
\begin{align*}
\gamma(C_{r}^{\text{Ga}})=\frac{4}{\pi}\operatorname{arcsin}\left(\frac{\sqrt{(1+r)(3+r)}-\sqrt{(1-r)(3-r)}}{4}\right),
\end{align*}
where $C_{r}^{\text{Ga}} \in \mathcal C_2$ is the Gaussian copula with correlation parameter $r \in [-1,1]$.
Since $\gamma(C_{-0.5}^{\text{Ga}})=-0.379$, we have that $P(-0.379) \in \mathcal R_d(\gamma)$.
Since $-\frac{1}{2}<-0.379 < -\frac{1}{3}$, it holds that $P(-0.379) \notin \mathcal P_d^{\rm{B}}$, and thus $P(-0.379)$ is the desired correlation matrix.
Note that the function $r \mapsto \gamma(C_{r}^{\text{Ga}})$ is continuous and strictly increasing on $(-1,1)$.
%Moreover, $\gamma(C_{-0.5}^{\text{Ga}})$,
Therefore, $P(r) \in \mathcal P_d$ for $-\frac{1}{2}\leq r < -0.379$, but they are not attainable as a $\gamma$-matrix of any Gaussian copula.
\end{example}

\section{Conclusion}\label{sec:conclusion}
Our article provides two main contributions. First we showed that measures of
concordance of the form $\kappa(U,V)=\rho(g_1(U),g_2(V))$ are exhausted by the
case when $g_1=g_2=G^\i$ for so-called concordance-inducing distributions
$G$. We relaxed the continuity assumption previously imposed
on $g_1,g_2$ in Hofert \& Koike (2019) and our more general result now covers
Blomqvist's beta, which is obtained for non-continuous functions $g_1,g_2$.  Our
characterization in terms of $G$ is useful for determining whether a given pair
of functions $g_1,g_2$ leads to a correlation-based measure of concordance. For
instance, by monotonicity of $X \mapsto G^\i(F_X(X))$, $X\sim F_X$, the correlation
coefficient transformed by the so-called wrapping function proposed by
Raymaekers \& Rousseeuw (2021) cannot be a measure of concordance despite its
efficiency and robustness.

As second contribution, we introduced generalized Gini's gamma and showed that
it can be represented as a mixture of correlation representations of measures of
concordance. As an application of this mixture representation, we derived lower
and upper bounds for the compatible set of generalized Gini's gamma, which
enable us to test whether a given square matrix is attainable as a matrix of
pairwise generalized Gini's gammas.

Although we derived bounds for the compatible set of generalized Gini's gamma,
the complete characterization of this compatible set is still an open problem.
Another interesting question for future research is whether any measure of concordance of degree one can admit certain correlation representation helpful to determine bounds of its compatible set.

\section*{ACKNOWLEDGEMENTS}
The authors are grateful to the editor, an associate editor, and two anonymous referees for their helpful and constructive comments.
Takaaki Koike was supported by JSPS KAKENHI Grant Number JP21K13275. 
Marius Hofert acknowledges financial support from the Natural Sciences and Engineering Research Council of Canada (RGPIN-2020-04897 and RGPAS-2020-00093).

\bibliographystyle{apalike}
%\bibliography{CorrelationRepresentation}

\end{document}